\documentclass[11pt]{amsart}
\usepackage{}
\usepackage{amsmath,amssymb,amsthm}
\usepackage[latin1]{inputenc}
\usepackage{version,tabularx,multicol}
\usepackage{graphicx,float}
\usepackage{color}

\usepackage{stmaryrd}

\newcommand{\reff}[1]{(\ref{#1})}

\theoremstyle{plain}
\newtheorem{theo}{Theorem}[section]

\newtheorem{cor}[theo]{Corollary}
\newtheorem{prop}[theo]{Proposition}

\newtheorem{lem}[theo]{Lemma}

\theoremstyle{remark}
\newtheorem{rem}[theo]{Remark}

\newcommand{\pthe}{\psi_{\theta}}
\newcommand{\uthe}[1]{u^{\theta}(\lambda, #1)}
\newcommand{\cthe}[1]{c^{\theta}(#1)}

\newcommand{\ebt}{\mathrm{e}^{2 \beta \theta t}}

\newcommand{\embt}{\mathrm{e}^{-2 \beta \theta t}}
\newcommand{\embs}{\mathrm{e}^{-2 \beta \theta s}}

\newcommand{\lam}{\lambda}

\newcommand{\ck}{{\mathcal K}}

\newcommand{\cm}{{\mathcal M}}

\newcommand{\Tau}{{\mathcal T}}

\newcommand{ \cz}{\mathcal Z}

\newcommand{\E}{{\mathbb E}}

\newcommand{\N}{{\mathbb N}}
\renewcommand{\P}{{\mathbb P}}

\newcommand{\R}{{\mathbb R}}

\newcommand{\rN}{{\rm N}}
\newcommand{\rP}{{\rm P}}
\newcommand{\rE}{{\rm E}}

\newcommand{\ind}{{\bf 1}}

\newcommand{\inv}[1]{\mathop{\frac{1}{ #1}}\nolimits}
\newcommand{\expp}[1]{\mathop {\mathrm{e}^{ #1}}}
\newcommand{\val}[1]{\mathop{\left| #1 \right|}\nolimits}

\title[Branching process with  non-neutral mutations]{A population model
  with non-neutral mutations using branching processes with immigration}

\date{\today}
\author{Hongwei Bi}
\address{
Hongwei Bi,
Université Paris-Est, CERMICS (ENPC), F-77455 Marne La Vallée, France
and 
School of Mathematical Sciences,
Beijing Normal University,
Beijing 100875, China.
}
\email{bih@cermics.enpc.fr}

\author{Jean-François Delmas}

\address{
Jean-Fran\c cois Delmas,
Université Paris-Est, CERMICS (ENPC), F-77455 Marne La Vallée, France.}

\email{delmas@cermics.enpc.fr}

\begin{document}

\subjclass[2010]{Primary: 60J80, 92D10; Secondary: 60G10, 60G55, 60J68, 92D25}

\keywords{non-neutral mutation, branching process, immigration,
  bottleneck, population model, genealogical tree, MRCA}

\begin{abstract}
  We  consider a stationary  continuous model  of random  size population
  with non-neutral mutations using  a continuous state branching process
  with non-homogeneous immigration.  We assume the type (or mutation) of
  the immigrants  is random given  by a constant mutation  rate
  measure. We determine  some genealogical properties
  of this process  such as: distribution of the time  to the most recent
  common  ancestor (MRCA),  bottleneck effect  at the  time to  the MRCA
  (which  might be drastic  for some  mutation rate  measures), favorable
  type for the MRCA, asymptotics of the number of ancestors.
\end{abstract}

\maketitle

\author{}
\date{\today}
\maketitle

\section{Introduction}

\subsection{Motivations}

Galton-Watson (GW)  processes are branching  processes modeling discrete
populations in discrete time. Since (non-degenerate) GW processes either
become  extinct  or  blow  up  at  infinity, one  needs  to  consider  a
stationary version of the GW  process, such as a sub-critical GW process
with immigration  to model the  stationary population. It is  well known
that  the  rescaled  limit  in  time  and  space  of  GW  processes  are
continuous-state   branching  (CB)   processes,   see  \textsc{Lamperti}
\cite{l:lsbp},  and  that  the  rescaled  limit  of  GW  processes  with
immigration are CB processes with immigration (CBI), see \textsc{Kawazu}
and \textsc{Watanabe} \cite{kw:bpirlt}.  Sub-critical CB process becomes
extinct a.s., and conditioning this  process not to become extinct gives
a CBI with a particular  immigration which is a natural continuous model
for  populations  with  stationary   random  size.   For  the  study  of
genealogical properties  of CBI we are interested  in, see \textsc{Chen}
and  \textsc{Delmas} \cite{cd:spstsbp}  and  \textsc{Bi} \cite{b:tmscp},
for a  more general immigration.  The  aim of this paper  is to consider
the  simplest CB  process (with  quadratic branching  mechanism)  and an
immigration taking  into account non-neutral mutations.   We shall prove
the  existence  of this  stationary  continuous  process  and give  some
genealogical properties, such as a  bottleneck effect at the time to the
most recent  common ancestor (TMRCA),  favorable mutation for  the most
recent common ancestor (MRCA) of the process.

\subsection{Constant size population models}

A large literature is devoted to the constant size population: Wright-Fisher
model  (discrete  time,  discrete  population),   Moran model    (continuous
time,  discrete  population) and   Fleming-Viot  process  (continuous  time,
continuum  for the  population).  Neutral  models can  be    described using
spatial    Fleming-Viot processes,    see \textsc{Dawson}
\cite{d:mvmp} and \textsc{Donnelly} and \textsc{Kurtz} \cite{dk:prmvpm}. 

Non-neutral mutation  models in stationary regime  have been considered
by  \textsc{Neuhauser} and  \textsc{Krone}  \cite{nk:gsms} for  discrete
population,   by   \textsc{Fearnhead}   \cite{f:psnpgm}   for   discrete
population  (possibly  with random  size)  in  continuous  time, and  by
\textsc{Stephens}    and    \textsc{Donnelly}    \cite{sd:ajpgms}    and
\textsc{Donnelly}  as  well   as  \textsc{Nordborg}  and  \textsc{Joyce}
\cite{dnj:lsmcnpgm}  for continuous  models  in which  the  type of  the
mutant does not depend on the type of the parent.  In \textsc{Fearnhead}
\cite{f:canl} and \textsc{Taylor} \cite{t:capwfd} the MRCA is  studied. 
In particular it is  shown in \cite{f:canl} that the  expected fitness of
MRCA is greater  than that of  a randomly chosen individual.

Notice  that the  non-neutral  models studied  by \textsc{Donnelly}  and
\textsc{Kurtz} \cite{dk:gpfvmsr} are  non-stationary.  Such models could
be  made stationary  by conditioning  on the  non-extinction of  all the
types, see \textsc{Foucart} and \textsc{Hénard} \cite{fh:scbpibfvpi} for
a  work   in  this  direction.   For  non-stationary   models  see  also
\textsc{Bianconi},  \textsc{Ferretti} and  \textsc{Franz} \cite{bff:ntb}
for constant size discrete population in discrete time.  In those latter
models the  non-neutral mutations are described using  an immigration at
constant rate but with various fitness.

\subsection{Random  size population models}

Another  large literature  is devoted  to random  size  population using
branching   processes.   Neutral   models  are   now  well   known,  see
\textsc{Bertoin}  \cite{b:saptpgwpnm,  b:lttabprnm}  for  GW  processes,
\textsc{Champagnat}, \textsc{Lambert}   and \textsc{Richard}  \cite{clm:bdpnm} 
for Crump-Mode-Jagers branching processes   and   \textsc{Abraham}   and
\textsc{Delmas}  \cite{ad:cbmcsbpui,ad:wdlcrtseppnm} with  a  CB process
presentation  in \cite{ad:cbmcsbpui}  and a  genealogical  tree approach
using  continuous random  tree  presentation in  \cite{ad:wdlcrtseppnm}.
Conditioning  on non-extinction  will  provide a  stationary model,  see
\cite{cd:spstsbp} in the CB process setting.

One can  also use multi-type processes for  non-neutral mutation models;
in all those models the rate  of mutation is proportional to the size of
the  current  population.   For   the  discrete  setting,  we  refer  to
\textsc{Athreay}   and  \textsc{Ney}   \cite{an:bp}  on   multi-type  GW
processes,    see    also    \textsc{Buiculescu}   \cite{b:qdmgwp}    or
\textsc{Nakagawa}  \cite{n:qpamgwpar}  for  sub-critical  multi-type  GW
processes conditioned on  non-extinction. Those latter processes provide
natural  stationary models.   Similar  results exist  for multi-type  CB
process   or  non-homogeneous   super-processes  (which   correspond  to
infinitely    many   types)    conditioned    on   non-extinction    see
\textsc{Champagnat}  and  \textsc{Roelly}  \cite{cr:ltcmdwpfd}  for  the
former case and  \textsc{Delmas} and \textsc{Hénard} \cite{dh:wdsds} for
the more  general latter case.  A non-stationary  model with non-neutral
mutation  is   also  given  in   \textsc{Abraham},  \textsc{Delmas}  and
\textsc{Hoscheit} \cite{adh:etiltvp},  with a model  of immigration with
(random)  increasing fitness.  In  this model,  the whole  population is
again a CB process.

\subsection{The model}
Our  present model  of  a CB  process  with non-homogeneous  immigration
follows  the  approach  of  \cite{f:psnpgm}  and  \cite{bff:ntb}  as  we
consider  non-neutral mutation  provided by  an immigration  at constant
rate.

For  simplicity, we shall  restrict our-self  to the  quadratic branching
mechanism    $\psi(\lambda)=\beta\lambda^2$,   with    $\beta>0$.    For
$\theta>0$,  let $Y^\theta=(Y^\theta_t,  t\geq 0)$  denote a  CB process
with                         branching                         mechanism
$\psi_\theta(\lambda)=\psi(\lambda+\theta)-\psi(\theta)$.   It  is  well
known  that   $Y^\theta$  is   stochastically  larger  than   $Y^q$  for
$\theta\leq q$.  In particular, we shall say that the type (or mutation)
$\theta$ is  more advantageous  than the type  $q$. We shall  consider a
stationary CBI  $(Z_t, t\in  \R)$ with non-homogeneous  immigration such
that  at rate  $2\beta \,dt$  there is  an immigration  of a  CB process
starting with an infinitesimal mass  and of type $\theta$, with $\theta$
chosen according  to a $\sigma$-finite measure  $\mu(d\theta)$.  We call
$\mu$ the mutation rate  measure.  

In \cite{cd:spstsbp},  the immigration was homogeneous  and the mutation
rate measure was a Dirac mass measure; we shall call this model CBI with
neutral mutations. In the framework of \cite{cd:spstsbp} the process $Z$
is distributed  as the initial CB process  conditioned on non-extinction
(or $Q$-process) under its stationary measure.  The immigration can also
be seen as  the descendants of an immortal  individual. This description
of the genealogy  using an immortal individual is in  the same spirit as
the   bottom   individual  in   the   modified   look-down  process   in
\cite{dk:prmvpm}. Even if this interpretation is  no more valid  in our
setting, we might keep the corresponding vocabulary as MRCA or TMRCA.

The mutation  rate measure allows to consider  non-neutral mutations. In
our  model different  CB  processes with  different branching  mechanism
coexist at the same time  and all mutations eventually die out. Contrary
to \cite{adh:etiltvp}, who  considered only advantageous mutations (that
is advantageous  immigration with rate  proportional to the size  of the
population), the type of the immigrants in our model is random and there
is no improvement of the type as time goes on. One of the advantages of
our model is that it has  a stationary version, which we shall consider.
Notice also  that the  size of the  population is random  (and different
from an homogeneous  CBI unless the mutation rate  measure is a constant
time a Dirac mass).

\subsection{Presentation of the results}

After some preliminaries on CB processes in Section \ref{sec:prelim}, we
define precisely our model  in Section \ref{sec:defZ}. In particular, we
give an integrability  condition on the mutation rate  measure $\mu$ for
the process $Z$ to be  well defined (Theorem \ref{conztfi}) and we check
that  $Z$  is  continuous   (Theorem  \ref{theo:cont-0}).  We  give  the
expectation of  $Z_t$ (Corollary  \ref{cor:EZ}) which might  be infinite
and characterize the mutation rate measure for which the population size
is always  strictly positive (Proposition  \ref{prop:Zt=0}). Notice only
this   case   is   biologically   meaningful.  We   also   give   (Lemma
\ref{lem:TMRCA})  the distribution  of time  to the  most  recent common
ancestor  (TMRCA) which  is seen  as the  first immigration  time  of an
ancestor of the current population living at time $t=0$.

We study the type $\Theta$ of the MRCA in fact that the type of the first 
immigrant  having descendants at time $t=0$ in Section \ref{sec:Q}. 
In particular, we get that if $\mu$ is a probability measure, 
then $\Theta$ is stochastically more favorable than the type of a random
immigrants given by $\mu$ (that is $\Theta$ is stochastically less than
$\Theta'$ with probability measure $\mu$). 

Using arguments close to
\cite{cd:spstsbp} we give  in Section      \ref{sec:bottle}
the distribution of the size $Z_A$ of the population at the
TMRCA (Proposition  \ref{prop:ZA})  and check that the size 
of the population at the  TMRCA  is stochastically  smaller 
than the size of the population at fixed time (which is the stationary measure).
 This can be interpreted as a bottleneck effect. 

 In Section  \ref{sec:N} we give (Lemma  \ref{lem:cvN/L}) the asymptotic
 number  of  immigrants  who  still  have  descendants  in  the  current
 population  at time  $t=0$.   In Section  \ref{sec:geneal},  we give  a
 precise  description of  the genealogical  structure of  the population
 relying on the  tree structure of the Brownian  excursion.  We study in
 Section \ref{sec:M}  the asymptotic number $M_s$ of  ancestors $s$ unit
 of time  in the past of  the current population.  In  particular we get
 (Proposition \ref{prop:cv-as-sM}) that $\beta M_s$ is of order $Z_0/s$,
 which  is   similar  to  the   CBI  case  with  neutral   mutation  see
 \cite{cd:spstsbp}   or  \textsc{Berestycki},   \textsc{Berestycki}  and
 \textsc{Limic}  \cite{bbl:lcscdi} for $\Lambda$-coalescent  models.  We
 also give the fluctuations (Theorem \ref{theo:fluct}) which are similar
 to the neutral case if the mutation rate behaves nicely.

Section \ref{sec:stable-mut} is devoted to the stable mutation rate:
\[
\mu(d\theta) = c  \theta^{\alpha - 1} \, \ind_{\{\theta > 0\}}\, d\theta,
\]
with $c>0$ and $\alpha\in (0,1)$. In this case we have $\rE[Z_0]=+\infty
$  and $\rE[Z_A]$  finite iff  $\alpha\in (1/2,  1)$. In  particular for
$\alpha\in (1/2,1)$  we have  a drastic bottleneck  effect as  the ratio
$\rE[Z_A]/\rE[Z_0]$  is  equal  to  $0$.   We  also  prove  (Proposition
\ref{prop:fitness}) that  the type of the MRCA  is (stochastically) more
advantageous  than the  type of  an individual  taken at  random  in the
current  population, see also  \cite{f:canl} for  similar behavior  in a
different model. We  conjecture this result holds for  any mutation rate
measure.   We  get  in  Section  \ref{sec:Nstable} that  the  number  of
families at $s$ unit of time  in the past for the model with non-neutral
mutations  behave as  $s^{-\alpha}$  and  that this  result  can not  be
compared  to the  neutral  case in  \cite{cd:spstsbp}  even with  stable
branching mechanism where the number of  families at $s$ unit of time in
the past is  of order $1/|\log(s)|$. Concerning the  fluctuations of the
number  of  ancestors, $M_s$,  we  get  that  results given  in  Theorem
\ref{theo:fluct}  holds  iff  $\alpha\in  (0,1/2)$  and  we  get  a
deterministic  limit for  $\alpha\in (1/2,1)$.  We interpret  this
latter phenomenon as a
 law of  large number effect from the  large number of small
populations generated by immigrants with very disadvantageous mutations
which is preponderant to the fluctuations of the number of ancestors in
each of the immigrant populations. 

\section{Preliminaries and notations}
\label{sec:prelim}

We consider a quadratic branching mechanism $\psi(\lam) = \beta \lam^2$
for some fixed $\beta>0$.
We will consider a family $(\pthe, \theta \geq 0)$ of (sub)-critical
branching mechanism defined by:
\[
\pthe(\lam) = \psi(\lam + \theta) - \psi(\theta) = 2 \beta \theta \lam +
\beta \lam^2. 
\]
For every  fixed $\theta\geq 0$, let  $\P_x^{\pthe}$ be the law  of a CB
process, $Y^{\theta}  = (Y_t^{\theta}, t  \geq 0)$, started at  mass $x$
with   branching   mechanism  $\pthe$.    Let   $\E_x^{\pthe}$  be   the
corresponding  expectation  and $\N^{\pthe}$  be  the canonical  measure
(excursion   measure)  associated   to   $Y^{\theta}$.   In   particular
$\N^{\pthe}$ is a $\sigma$-finite measure on the set $D_0$ of continuous
functions from $(0, \infty)$ to $[0, \infty)$ having zero as a trap (for
a function $f$, this means $f(s)=0$ implies $f(t)=0$ for all $t\geq s$).
According    to   \cite{ad:cbmcsbpui},   see    also   \textsc{Abraham},
\textsc{Delmas} and \textsc{Voisin}  \cite{adv:plcrt}, it is possible to
define the  processes $(Y^\theta,  \theta\geq 0)$ on  the same  space so
that a.s. $Y^{\theta_1} \geq  Y^{\theta_2}$, for any $0\leq \theta_1\leq
\theta_2$.

When there is no confusion we shall write $Y$ for $Y^\theta$ that is for
example $ \E^{\psi_\theta}_{x}[F(Y)]$ or $ \N^{\psi_\theta}[F(Y)]$
instead of $ \E^{\psi_\theta}_{x}[F(Y^\theta)]$ or $
\N^{\psi_\theta}[F(Y^\theta)]$. We recall some well known results on
quadratic CB processes. 
For every $t \geq 0$ and $\lam >  -2\theta/(1- \expp{-\theta t})$, we have: 
\begin{equation} 
\label{eq:laplace_csbp}
\E_x^{\pthe}\left[\expp{-\lam Y_t}\right] = \expp{-x \uthe{t}},
\end{equation}
where
\[
\uthe{t} = \N^{\pthe} [1- \expp{-\lam Y_t}]  
= \frac{2 \theta \lam}{(2 \theta + \lam) \ebt - \lam}\cdot
\]
Notice that $u^\theta$ satisfies the backward and forward equations for
$\lambda\geq 0$ and $t\geq 0$:
\[
\partial_t \uthe{t} = -\pthe(\uthe{t}),\quad
\partial_t \uthe{t} = -\pthe(\lam)\,\partial_\lam \uthe{t},
\]
with initial conditions $u^\theta(\lambda,0)=\lambda$ and
$u^\theta(0,t)=0$. 
It is easy to deduce that for $t\geq 0$:
\begin{equation}
   \label{eq:mean}
\N^{\pthe} [Y_t] = \embt \quad\text{and}\quad
\N^{\psi_\theta} \left[Y_t \expp{-\lambda Y_t}\right]
= \frac{\expp{-2\beta\theta t}}{\left(1+ \lambda  \Delta^\theta _t \right)^2},
\end{equation}
where we set for $t\geq 0$:
\begin{equation}
   \label{eq:def-D}
\Delta_t^\theta=\frac{1-\expp{-2\beta\theta t}}{2\theta}\cdot
\end{equation}
It is easy to get that for $t\geq 0$ and $\lambda\geq 0$:
\begin{equation} 
\label{int-u}
\beta \int_0^t \uthe{r}\, dr = \log(1 + \lambda \Delta_t^\theta )
\quad\text{and}\quad
\beta \int_0^\infty \uthe{r}\, dr = \log\left(1 + \frac{\lam}{2
    \theta}\right). 
\end{equation}
Let  
\begin{equation}
   \label{eq:lifetime}
\zeta=\inf\{t>0; Y_t=0\}
\end{equation} 
be the lifetime of $Y$ and set $\cthe{t} =
\N^{\pthe} [\zeta > t] $. Then we have: 
\begin{equation}
   \label{eq:c_q}
\cthe{t} = \lim_{\lam \to \infty} \uthe{t}= \frac{2 \theta}{\ebt - 1}=
\frac{\expp{-2\beta\theta     t}}{\Delta_t^\theta}\cdot
\end{equation} 
Notice that for $t,s>0$, we have $c^\theta(t+s)=u^\theta(c^\theta(t),
s)$. We also have for $t>0$:
\begin{equation} 
\label{int-c}
\beta \int_t^{\infty } c^\theta(r)\, dr  = -\log \left( 1- \expp{-2\beta\theta t}\right). 
\end{equation}

\section{Definition and properties of the total size process}
\label{sec:defZ}

For a Borel measure $\mu$ and a measurable non-negative function $f$ defined
on the same space, we will write $\langle \mu,f \rangle=\int
f(x)\,\mu(dx)$.

Let $\mu$ be a non-zero Borel $\sigma$-finite measure on $(0,+\infty )$,
which we  shall call  a mutation rate  measure.  Consider under  $\rP$ a
Poisson point measure (PPM) on $\R \times (0, \infty) \times D_0$,
$$
\sum_{i \in I} \delta_{(t_i, \theta_i, Y^i)}(dt, d \theta, dY),
$$
with intensity $2 \beta dt  \mu(d \theta) \N^{\pthe}[dY]$.  Let $\rE$ be
the  expectation corresponding  to the  probability measure  $\rP$.  For
$i\in I$,  we shall call $Y^i$  a family, $\theta_i$ its  type (or
mutation) and $t_i$
its birth  time.  Define the super-process $\cz=(\cz_t,  t\in \R)$ by:
$$
\cz_t(d \theta) = \sum_{i\in I} Y_{t - t_i}^i \delta_{\theta_i}(d \theta)
$$
with the convention  that $Y_t^i = 0$ for $t  < 0$ and $\delta_{\theta}$
denotes  the  Dirac  mass  at  $\theta$.  By  construction  $\cz$  is  a
stationary Markovian $\sigma$-finite measure-valued process.  We shall
consider the corresponding total size process $Z=(Z_t, t\in  \R)$
defined by:
\[
Z_t = \langle \cz_t, 1
\rangle  = \sum_{i\in I} Y_{t - t_i}^i = \sum_{t_i<t} Y_{t - t_i}^i . 
\]
Notice that  $Z$  is stationary  but  it is  not
Markovian unless $\mu$ is a constant times a Dirac mass. The process $Z$
is a CB process with a non-homogeneous immigration. It will represent the
evolution of a random size population with non-neutral mutations in a
stationary regime. The genealogy of $Z$ will be defined in Section
\ref{sec:geneal}.

First we will consider  the condition on $\mu$ such that $Z$ is well
defined. 
\begin{theo} \label{conztfi}
Let $t\in \R$. The random variable $Z_t$ is finite a.s. 
if and only if the following conditions are satisfied:
\begin{equation}
\label{eq: con-mu}
\int_{0+} |\log \theta | \,\mu(d \theta) < \infty\quad\mbox{and}
\quad\int^{+\infty} \frac{\mu(d \theta)}{\theta} < \infty.
\end{equation}
The distribution of $Z_t$ is characterized by its Laplace
transform, for $\lambda\geq 0$:
\begin{equation}
   \label{eq:lap-Z}
\rE[\expp{-\lam Z_t}]=\exp\left(-2    \int_0^{\infty}
  \mu(d\theta)\, \log\left(1+\frac{\lambda}{2\theta}\right)\right). 
\end{equation}

\end{theo}

\begin{proof}
By the exponential formula, one obtains that, for $F$ non-negative measurable,
$$
\rE\left[\exp\left(-\sum_{i \in I} F(t_i, \theta_i, Y^i)\right)\right]
= \exp\left(-2 \beta \int_0^{\infty} dt \int_0^{\infty} \mu(d\theta)\,
\N^{\pthe}[1- \expp{-F(t, \theta, Y)}]\right).
$$
Since $Z_t = \langle \cz_t, 1 \rangle = \sum_{t_i \le t} Y_{t - t_i}^i$, 
we have using \reff{int-u}:
\begin{align*}
\rE[\expp{-\lam Z_t}]
&=\exp\left(-2 \beta \int_0^{\infty} ds
\int_0^{\infty} \mu(d\theta)\,\N^{\pthe}[1- \expp{-\lam Y_s}]\right)\\
&=\exp\left(-2  \beta \int_0^{\infty} ds \int_0^{\infty}
  \mu(d\theta)\,\uthe{s}\right)\\
&= \exp\left(-2    \int_0^{\infty}
  \mu(d\theta)\, \log\left(1+\frac{\lambda}{2\theta}\right)\right). 
\end{align*}
Letting $\lam \to 0$ entails that
\begin{equation} \label{equizf}
\rP(Z_t < \infty)
= 1 \quad \Leftrightarrow \quad \lim_{\lam \to 0}
\int_0^{\infty}
  \mu(d\theta)\, \log\left(1+\frac{\lambda}{2\theta}\right) = 0.
\end{equation}
The right hand side of \eqref{equizf} is equivalent to the existence of
some $\lam > 0$, such that:
\begin{equation} \label{con-int-u}
\int_0^{\infty}
  \mu(d\theta)\, \log\left(1+\frac{\lambda}{2\theta}\right)< \infty.
\end{equation}
As  $\log(1 + \lambda/2\theta )$ is equivalent to $|\log \theta |$ 
(resp. $\lambda/2\theta$) as $\theta$ goes to $0+$
(resp. $+\infty$), we deduce that \eqref{con-int-u} holds if and only if 
 \eqref{eq: con-mu} holds.
\end{proof}

Before giving other properties of the process $Z$, we shall study the
time $A$ to the first immigration time of an ancestor  (or equivalently
the TMRCA)  of the current
population 
living at time $0$ which is  defined as:
\begin{equation}
   \label{eq:def-A}
A= \sup \{|t_i|; i\in I \text{ and } Y_{-t_i}^i > 0\}= \sup \{|t_i|;
i\in I \text{ and } t_i<0<t_i+ \zeta_i\},
\end{equation}
with $\zeta_i$ the lifetime (see definition \reff{eq:lifetime}) of
$Y^i$. 

\begin{lem}
   \label{lem:TMRCA}
We have for all $t\geq 0$: 
\[
\rP(A < t) =\exp\left(2\int_0^\infty \log(1-\expp{-2 \beta \theta
    t})\,\mu(d \theta)\right).
\]
Under conditions \reff{eq: con-mu}, we get that  $A$ is a.s. finite. 
\end{lem}
\begin{proof}
   The property of the Poisson random measure implies that for $t\geq
   0$: 
\begin{align*}
\rP(A < t)
&= \rP(\forall\, i \in I, t_i \geq  -t \text{ or } \zeta_i +t_i< 0)\\
&= \exp\left(-2 \beta \int_t^\infty ds \int_0^\infty
  \mu(d\theta)\,\N^{\pthe}[\zeta<s ]\right)\\
&= \exp\left(-2 \beta \int_0^\infty \mu(d\theta) \int_t^\infty \cthe{s}\,ds\right)\\
&= \exp\left(2\int_0^\infty \log(1-\expp{-2 \beta \theta t})\,\mu(d \theta)\right),
\end{align*}
where  we used  \reff{int-c}  for the  last  equality. Under  conditions
\reff{eq:  con-mu},  we get  that  $\int_0^\infty \log(1-\expp{-2  \beta
  \theta t})\,\mu(d \theta)$  is finite for any $t>0$,  which
implies, thanks to 
dominated convergence, that $\lim_{t\rightarrow+\infty } \rP(A<t)=1$ that
is $A$ is a.s. finite.
\end{proof}

\begin{theo}
   \label{theo:cont-0}
Under conditions \reff{eq: con-mu}, the process $Z$ is continuous. 
\end{theo}
In particular, we deduce that under conditions \reff{eq: con-mu}, $\cz$
is a stationary Markov process with values in the set of finite measures
on $\R_+$. 

\begin{proof}
To  prove  the  continuity  of  the  process  $Z$,  we  notice  that  by
stationarity,  we  just   need  to  prove  the  continuity   of  $Z$  on
$[0,1]$.  Let  $c>0$ be  a  finite  constant  and consider  the
truncated process
$Z^c=(Z^c_t, t\in [0,1])$ defined by:
\[
Z^c_t= \sum_{i\in I} Y_{t - t_i}^i \ind_{\{ t_i\geq - c\}}. 
\]
Notice that $Z^c$ and $(Z_t, t\in [0,1])$ coincide on $\{A\leq
c\}$. Since $A$ is a.s. finite, to get the continuity of $Z$ on $[0,1]$,
we just need to prove that $Z^c$ is continuous. We shall check the
Kolmogorov criterion for $Z^c$. 

Let $\lambda\geq 0$ and $\gamma\geq 0$, $0\leq  s\leq t\leq 1$. 
We have:
\begin{align*}
   \rE\left[\expp{-\lambda Z^c_t -\gamma Z^c_s}\right]
&= \rE\left[\expp{-\sum_{-c\leq t_i \leq s}(\lam Y_{t - t_i}^i +\gamma
    Y_{s - t_i}^i) } \right]
\rE\left[\expp{-\lam \sum_{s < t_i \leq t}Y_{t - t_i}^i }\right]\\
&= \expp{-2 \beta \int_{-c}^s dr \int_0^\infty \mu(d\theta) 
\N^{\pthe}[1 - \expp{-\lam Y_{t - r} -\gamma  Y_{s - r}}]}
\expp{-2 \beta \int_s^t dr \int_0^\infty \mu(d\theta) \N^{\pthe}[1 -
  \expp{-\lam Y_{t - r}}]}\\ 
&= \expp{-2 \beta \int_0^{c+s} dr \int_0^\infty \mu(d\theta) 
u^\theta (u^\theta (\lam, t - s) +\gamma, r)} 
\expp{-2 \beta \int_0^{t - s} dr \int_0^\infty \mu(d\theta) u^{\theta}(\lam, r)}\\
&= \exp{-2 \int_0^\infty \mu(d\theta) \left(
\log(1+(u^\theta (\lam, t - s) +\gamma) \Delta^\theta_{c+s}) 
+ \log(1+\lambda \Delta^\theta_{t-s} )
\right)}\\
&= \exp{-2 \int_0^\infty \mu(d\theta) 
\log\left(1+\lambda \Delta^\theta_{t-s}+ (\lambda(1- 2\theta
\Delta^\theta_{t-s}) +\gamma(1+\lambda  \Delta^\theta_{t-s}) )
\Delta^\theta_{c+s}\right)}, 
\end{align*}
where we used \reff{int-u} for the fourth equality, 
and the equality $u^\theta (\lam, t - s)= \lambda (1-2\theta
\Delta^\theta_{t-s})/(1+\lambda  \Delta^\theta_{t-s}) $ (see 
\reff{eq:def-D}) for
the fifth. Notice that for fixed $r>0$, there exists a constant $C_r>0$
such that for all $\theta>0$:
\begin{equation}
   \label{eq:Dcr}
0\leq \Delta^\theta_r\leq \frac{C_r}{\theta+1}
\quad\text{and recall}\quad
 1-2\theta \Delta^\theta_r= \expp{-2\beta \theta r}.
\end{equation}
Therefore,  there exists a constant $c_1\geq 1$ such that for $\lambda,\gamma\in \R$:
\[
\val{\lambda \Delta^\theta_{t-s}+ (\lambda(1- 2\theta
\Delta^\theta_{t-s}) +\gamma(1+\lambda  \Delta^\theta_{t-s}) )
\Delta^\theta_{c+s}}
\leq  \frac{c_1}{1+\theta} (\val{\lambda}+\val{\gamma}+\val{\lambda\gamma}).
\]
We deduce that under conditions \reff{eq: con-mu}, the
function
\[
(\lambda,\gamma) \mapsto \int_0^\infty \mu(d\theta) 
\log\left(1+\lambda \Delta^\theta_{t-s}+ (\lambda(1- 2\theta
\Delta^\theta_{t-s}) +\gamma(1+\lambda  \Delta^\theta_{t-s}) )
\Delta^\theta_{c+s}\right)
\]
is analytic in  $(\lambda,\gamma)$ in a neighborhood of $0$ for example
on $\{(\lambda,\gamma); \val{\lambda}+\val{\gamma}\leq 1/4c_1\}$. Taking
$\gamma=-\lambda$, this
implies that for $\val{\lambda}\leq 1/8c_1$, we have:
\[
   \rE\left[\expp{-\lambda (Z^c_t -Z^c_s)}\right]
= \exp{-2 \int_0^\infty \mu(d\theta) 
\log\left(1+\lambda \Delta^\theta_{t-s}(1-  2\theta\Delta^\theta_{c+s})
  - \lambda^2 
\Delta^\theta_{t-s}
\Delta^\theta_{c+s}\right)}. 
\]
Using \reff{eq:Dcr}, an easy
computation yields that there exists a constant $c_2$ such that: 
\begin{multline*}
\rE\left[(Z^c_t -Z^c_s)^4\right]\\
\leq  c_2\left( \left(\int_0^\infty \mu(d\theta)
    \Delta^\theta_{t-s}\expp{-2\beta\theta(c+s)}\right)^4
    +\left(\int_0^\infty \frac{\mu(d\theta)}{1+\theta}
    \Delta^\theta_{t-s}\right)^2+\int_0^\infty \frac{\mu(d\theta)}{1+\theta}
     \left(\Delta^\theta_{t-s}\right)^2\right) .
\end{multline*}
Then using that $\val{ \Delta^\theta_{t-s}}\leq  \beta(t-s)$, 
we get there exists a constant $c_3$ such that: 
\[
\rE\left[(Z^c_t -Z^c_s)^4\right]\leq c_3 \val{t-s}^2.
\]
This gives the Kolmogorov criterion for $Z^c$. Thus $Z^c$ is continuous,
which ends the proof.
\end{proof}

We give the first moment of $Z$. 

\begin{cor}
\label{cor:EZ}
Under conditions \reff{eq: con-mu}, we have for $t\in \R$:
\begin{equation} \label{expe-zt}
\rE [Z_t] = \int_0^\infty \frac{\mu(d \theta)}{\theta} \in [0,\infty].
\end{equation}
\end{cor}

\begin{proof}
Using \reff{eq:lap-Z}, we get:
\[
\rE [Z_t] 
= 2 \beta \int_0^\infty ds \int_0^\infty \mu(d\theta)\,\N^{\pthe}[Y_s]
= 2 \beta \int_0^\infty ds \int_0^\infty \mu(d\theta)\,\expp{-2 \beta \theta s}
= \int_0^\infty \frac{\mu(d \theta)}{\theta} \cdot
\]
\end{proof}

We give  a criterion for $Z$ to  reach $0$. 
See also \textsc{Foucart} and \textsc{Bravo} \cite{fb:lecsbpi} for such
a criterion for CBI.

\begin{prop}
\label{prop:Zt=0}
  Under  conditions  \reff{eq:  con-mu},  we  have $\{t;  Z_t  =  0\}  =
  \emptyset$ a.s. if and only if
\begin{equation}
   \label{eq:Z=0}
\int_0^1 dt\, \exp\left(-2 \int_0^\infty \log(1 - \embt)\,\mu(d
  \theta)\right)  = \infty.
\end{equation}
In particular, $\{t;  Z_t  =  0\}  =
  \emptyset$ a.s. if $\langle \mu, 1 \rangle>1/2$ and with
  probability strictly positive $\{t;  Z_t  =  0\}  \neq 
  \emptyset$ if $\langle \mu,1 \rangle<1/2$. 
\end{prop}
\begin{proof}
   Recall that $\zeta_i$ is the lifetime of $Y^i$. By using Theorem 2 in
\textsc{Fitzsimmons}, \textsc{Fristedt} and \textsc{Shepp}   \cite{ffs:srnlurci}, 
we can derive that $\{t; Z_t = 0\} = \emptyset$ a.s. if and only if:
\[
\int_0^1 \exp\left(2 \beta\int_t^\infty ds \int_0^\infty \mu(d
  \theta) \N^{\psi_\theta} [\zeta>s] \right) dt = \infty. 
\]
Thanks to \reff{eq:c_q} and \reff{int-c}, 
this last condition is equivalent to \reff{eq:Z=0}. 

If $\langle \mu, 1 \rangle>1/2$, then there exists $\theta_0\in
(0,+\infty )$ such that 
$B=\int_0^{\theta_0}  \mu(d\theta)>1/2$. Then, we have:
\[
-2 \int_0^\infty \log(1 - \embt)\,\mu(d
  \theta)\geq 
-2 \int_0^{\theta_0} \log(1 - \embt)\,\mu(d
  \theta)
\geq  -2 B \log (1-\expp{-2\beta\theta_0 t}) .
\]
As $2B>1$, we deduce:
\[
   \int_0^1 dt \, \exp\left(-2 \int_0^\infty \log(1 - \embt)\,\mu(d
  \theta)\right) 
 \geq \int_0^1 (1-\expp{-2\beta\theta_0 t}) ^{-2B} \, dt=+\infty . 
\]
Thus a.s. $\{t;  Z_t  =  0\}  =
  \emptyset$.

If $\langle \mu, 1 \rangle<1/2$, then, as $1-\expp{-x} \geq x/2$
for $x\in [0,1]$,  we have for $t\in (0, 1/2\beta]$:
\begin{align*}
   -2 \int_0^\infty \log(1 - \embt)\,\mu(d
  \theta)
&\leq   -2 \int_0^1 \log(1 - \embt)\,\mu(d
  \theta)  -2  \log (1-\expp{-2\beta t}) \int_1^\infty \mu(d\theta)\\
&\leq   -2 \int_0^1 \log(\beta\theta t)\,\mu(d
  \theta)  -2  \log (1-\expp{-2\beta t}) \int_1^\infty \mu(d\theta)\\
&= C -2\log (t) \int_0^1 \mu(d\theta) -2  \log (1-\expp{-2\beta t})
\int_1^\infty \mu(d\theta),
\end{align*}
where $C$ is a finite constant thanks to \reff{eq: con-mu}. We deduce
that for $\varepsilon>0$ small enough:
\[
 \int_0^\varepsilon dt \, \exp\left(-2 \int_0^\infty \log(1 - \embt)\,\mu(d
  \theta)\right) 
\leq \int_0^\varepsilon  dt\, t^{-2 \int_0^1 \mu(d\theta)} 
(1-\expp{-2\beta t})^{ -2  \int_1^\infty \mu(d\theta)}   \expp{C}<+\infty ,
\]
as $\langle \mu,1  \rangle<1/2$. 
This implies that with strictly positive probability  $\{t;  Z_t  =  0\}
\neq 
  \emptyset$.
\end{proof}

\section{Type of the MRCA}
\label{sec:Q}
\textbf{We  assume that conditions \reff{eq: con-mu} hold.}

Because of the stationarity, we shall focus on the MRCA 
of the current population living at time $0$. Recall the TMRCA
is given by \reff{eq:def-A}. We set $i_0\in I$ the (unique) index $i$ such that
$A=-t_i$. We shall say that $Y^{i_0}$ is the oldest family. 
We define the type of the MRCA that is of the oldest immigrant family as:
\[
\Theta=\theta_{i_0}.
\]
We give the joint distribution of the TMRCA  and the type of the MRCA. 

\begin{lem} 
We have for every $t \in \R, \theta > 0$,  
\begin{equation} \label{eq:joint-athe}
\rP (A \in dt, \Theta \in d \theta)
=  \frac{4\beta  \theta}{\expp{2 \beta \theta t} - 1}
\exp\left(2 \int_0^\infty \log(1 - \expp{-2 \beta \theta'
    t})\mu(d\theta')\right)\,dt \mu(d\theta). 
\end{equation}
\end{lem}
\begin{proof}
For $f$ non-negative measurable, 
we  get:
\begin{align*}
\rE [f(A, \Theta)]
&= \rE\bigg[\sum_{i\in I} f(-t_i, \theta_i)\,
\ind_{\big\{Y_{-t_i}^i > 0, \sum_{t_j < t_i}\ind_{\{Y_{-t_j}^j > 0\}} = 0\big\}}\bigg]\\
&= 2 \beta \int_0^\infty ds \int_0^\infty\mu(d\theta)\,
f(s, \theta)\,\N^{\pthe}[Y_s > 0]\,\rP(A < s)\\
&= 2 \beta \int_0^\infty ds \int_0^\infty\mu(d\theta)\,f(s, \theta)\,\cthe{s}\,\rP(A < s).
\end{align*}
We deduce that $\rP (A \in dt, \Theta \in d \theta)= 2 \beta
\cthe{t}\,\rP(A < t)\, dt\mu(d\theta)$.
Then, using Lemma \ref{lem:TMRCA},  it follows that:
\begin{equation*}
\rP (A \in dt, \Theta \in d \theta)
= 2 \beta\,dt \mu(d\theta) \frac{2 \theta}{\expp{2 \beta \theta t} - 1}
\exp\left(2 \int_0^\infty \log(1 - \expp{-2 \beta \theta' t})\mu(d\theta')\right).
\end{equation*}
\end{proof}

Using Lemma \ref{lem:TMRCA}, we can derive the distribution
$\mu_t^{\text{MRCA}}$ of the type
of MRCA given the TMRCA being equal to $t$:
\[
\mu_t^{\text{MRCA}}(\Theta \in d \theta) 
= \rP(\Theta \in d \theta| A \in dt)
= \frac{\theta(\ebt - 1)^{-1}}
{\int_0^\infty\theta'(\expp{2 \beta \theta' t}-1) ^{-1} \,  \mu(d\theta')}
  \, \mu(d\theta).
\]
Notice  that  the  function   $\theta  \mapsto  \theta(\ebt  -  1)^{-1}$
decreases  to 0 as  $\theta$ increases  to $+\infty$.   Intuitively, the
distribution of  the type  of the MRCA  is more  likely to focus  on the
favorable $\theta$  (that is $\theta$  small which corresponds  to large
population) than that on the  $\theta$ large (which corresponds to small
population).   In particular  if  $\mu$ is  a  probability measure  then
$\mu_t^{\text{MRCA}}$ is stochastically  smaller that $\mu$.  This means
the type  of the  MRCA (given $\{A=t\}$)  is stochastically  less, which
means  stochastically  more  favorable,   than  the  type  of  a  random
immigrant.  Notice  also  that $\mu_t^{\text{MRCA}}$  is  stochastically
decreasing with  $t$ which means  that the oldest family,  $Y^{i_0}$, is
stochastically increasing with $|t_{i_0}|$.

\section{Bottleneck effect}
\label{sec:bottle}
\textbf{We  assume that conditions \reff{eq: con-mu} hold.}

We consider $Z_{-A}$, which we shall denote $Z^A$, 
the size of the population at the TMRCA: 
\[
Z^A = Z_{-A}=\sum_{i\in I} Y_{-A - t_i}^i=\sum_{t_i < -A} Y_{-A - t_i}^i.
\]
Let $Z^O= Y_{-t_{i_0}}$ be the size of the old family at time $0$ and 
 $Z^I = Z_0 - Z^O$ be  the size  of the population at time
$0$ not belonging to the old family. Following Theorem 4.1 in \cite{cd:spstsbp}, 
it is easy to get the following result. 

\begin{lem} \label{lem:jd}
The joint distribution of $(Z^A, Z^O, Z^I, A, \Theta)$ is
characterized by:
for $\lam, \gamma, \eta \in [0, +\infty )$, and $t, \theta \in
(0,+\infty ) $,
\begin{multline*}
\rE[\exp(-\lambda Z^A-\gamma Z^I -\eta Z^O); A \in dt, \Theta \in d \theta]\\
= 2 \beta\,dt \mu(d\theta)\,\, (\cthe t - u^\theta(\eta, t))
 \exp\left(-2 \beta \int_0^t ds \int_0^\infty u^{\theta'}(\gamma,
  s)\,\mu(d\theta') \right)\\
\exp\left(
-2 \beta \int_0^\infty ds \int_0^\infty u^{\theta'}(\lam + c^{\theta'}(t), s)\,\mu(d\theta')\right).
\end{multline*}
\end{lem}
We deduce the following result. 
\begin{lem}
Conditionally on $A$, $(Z^O, \Theta)$, $Z^I$ and $Z^A$ are independent.
\end{lem}
Now we concentrate on the population size at the MRCA. Recall
$\Delta^\theta_t$ defined in \reff{eq:def-D}. 
\begin{prop}
   \label{prop:ZA}
Let $t \in (0,+\infty ) $. We have for $\eta\geq 0$:
\begin{equation} 
\label{eq:EexpZA}
\rE[\expp{-\eta Z^A}| A = t] 
=\exp{\left(-2  \int_0^\infty 
\log(1+\eta \Delta_t^\theta)
\, \mu(d\theta)\right)},
\end{equation}
and
\begin{equation} 
\label{eq:EZA}
\rE[Z^A| A = t] = 2 \int_0^\infty \Delta^\theta_t \,\mu(d
\theta)<+\infty .
\end{equation}
Furthermore conditionally on A (or not), $Z^A$ is stochastically
smaller than $Z_0$, that is for all $z>0$:
\begin{equation}
   \label{eq:order}
\rP(Z^A\leq z|A=t) \geq  \rP(Z_0\leq z)
\quad \text{and}\quad
\rP(Z^A\leq z) \geq  \rP(Z_0\leq z). 
\end{equation}
\end{prop}
The fact that 
$Z^A$ is stochastically
smaller than $Z_0$
corresponds to the bottleneck effect. 
\begin{proof}
Using \eqref{eq:joint-athe}, we get: 
\begin{equation} \label{conzaa} 
\rE[\expp{-\eta Z^A}| A = t] 
=\frac{\exp\left(-2 \beta \int_0^\infty ds
\int_0^\infty u^{\theta} (\eta + c^{\theta}(t), s)\, \mu(d \theta)\right)}
{\exp\left(-2 \beta \int_0^\infty ds \int_0^\infty u^{\theta} ( c^{\theta}(t), s)\,\mu
    (d \theta)\right)}\cdot 
\end{equation}
Then, using \reff{int-u}, \reff{int-c} and
\reff{eq:def-D}, it is easy to get \reff{eq:EexpZA}. This readily
implies \reff{eq:EZA}. 

We now prove the stochastic order. First notice that
$\Delta_t^\theta\leq 1/(2\theta)$. We deduce  that for all $\eta\geq
0$, we have:
\[
\rE[\expp{-\eta Z^A}| A = t] \geq \rE[\expp{- \eta Z_0}].
\]
This means that  $Z^A$ is  smaller than $Z_0$ in the Laplace
transform order. 
We will however prove the stronger result on the stochastic order.

We  deduce  from   \reff{eq:EexpZA}  that  conditionally  on  $\{A=t\}$,
$Z^A$  is   distributed  as  $\sum_{i\in   I}  Y^{i,1}_{-t_i}$,  with
$\cz^1=\sum_{i\in I} \delta_{Y^{i,1}, t_i}$ a PPM with intensity $2\beta
dt  \int_0^\infty  \mu(d\theta) \N^{\psi_{1/(2\Delta_t^\theta)}}  [dY]$.
As recalled in Section \ref{sec:prelim}, it is possible to define on the
same space two CB processes $Y^1$ and $Y^2$ such that $Y^1\leq Y^2$ a.e. and $Y^1$
(resp.  $Y^2$)  is  distributed under  $\N^{\psi_{1/(2\Delta_t^\theta)}}
[dY]$    (resp.      $\N^{\psi_\theta}    [dY]$)    since    $\theta\leq
1/(2\Delta_t^\theta)$.   We deduce  that  $\cz^1$ can  be  defined on  a
possible  enlarged  space  so  that there  exists a  PPM $\cz^2=\sum_{i\in  I}
\delta_{Y^{i,2},  t_i}$  with  intensity $2\beta  dt \int_0^\infty
\mu(d\theta) \N^{\psi_{\theta}} [dY]$ and  such that a.s.  for all $i\in
I$,  $Y^{i,1}\leq  Y^{i,2}$.  This  implies  that  a.s.  $\sum_{i\in  I}
Y^{i,1}_{-t_i}\leq  \sum_{i\in  I}  Y^{i,2}_{-t_i}$. As  $\sum_{i\in  I}
Y^{i,2}_{-t_i}$ is distributed as $Z_0$, we deduce that 
$Z^A$ (conditionally  on  $\{A=t\}$)  is stochastically less than
$Z_0$. This gives the first part of \reff{eq:order}.
By integrating the first part of \reff{eq:order} with respect to
$\rP(A\in dt)$, we
deduce the second part of \reff{eq:order}.
\end{proof}

As a direct consequence of \reff{eq:EZA}
we can compute the expectation of $Z^A$, which will be used in Section
\ref{sec:stable-mut}. 
 
\begin{lem} We have:
\begin{equation} 
\label{expeza}
\rE[Z^A] = 2 \beta \int_0^\infty dt \int_0^\infty \mu(d\theta)
\frac{2\theta}{\ebt-1}  
\int_0^\infty \mu(d\theta') \frac{1 - \expp{- 2 \beta \theta' t}}{\theta'}
\expp{2 \int_0^\infty \log(1 - \expp{-2 \beta \theta'' t}) \mu(d\theta'')}.
\end{equation}
\end{lem}

\section{Asymptotics for the number of families and ancestors}
\textbf{We  assume that conditions \reff{eq: con-mu} hold.}

\subsection{Asymptotics for the number of families} 
\label{sec:N}
For $s>0$, let $N_s$ be the number of families at time $-s$ which are still
alive at time $0$:
\[
N_s = \sum_{i \in I} \ind_{\{t_i < -s, \zeta_i > -t_i\}} =\sum_{i \in I}
\ind_{\{t_i<-s, Y_{-t_i}^i>0\}}.
\]
We set:
\begin{equation}
   \label{eq:def-L}
\Lambda(s)=-2 \int_0^\infty \mu(d \theta)\,\log(1 - \embs).
\end{equation}
We have the following result.
\begin{lem}
   \label{lem:cvN/L}
We have a.s.:
\[
\lim_{s \downarrow 0+} \frac{N_s}{\Lambda(s)} = 1.
\]
\end{lem}
\begin{proof}
   Notice that 
 $N_s$ is by construction a Poisson random variable with intensity
\[
2 \beta \int_{-\infty}^{-s} dr \int_0^\infty \mu(d \theta)\,\N^{\pthe}[\zeta > -r]
= -2 \int_0^\infty \mu(d \theta)\,\log(1 - \embs)=\Lambda(s),
\]
where we used  \reff{eq:c_q} and \reff{int-c} for the first equality. 
As $s \to 0+$, $\Lambda(s)$ goes to infinity.
Then notice that $(N_{\Lambda^{-1}(s)},s\geq 0)$ is a Poisson process
with parameter 1.  We deduce the
result from  the strong law of large numbers for L\'evy processes.
\end{proof}

\subsection{The genealogical tree}
 \label{sec:geneal}
 In order to consider the number  of ancestors $M_s$ at time $-s$ of the
 current  population  living at  time  $0$,  we  need to  introduce  the
 genealogical tree for a CB process, see \textsc{Le Gall} \cite{lg:iert}
 or \textsc{Duquesne} and \textsc{Le Gall} \cite{dlg:rtlpsbp}. Since the
 branching mechanism  is quadratic, we  will code the  genealogical tree
 using Brownian excursion.  Let $ W =  (W_t, t \in \R_+) $ be a Brownian
 motion.   We consider the  Brownian motion  $W^\theta=(W_t^\theta, t\in
 \R_+)$  with negative  drift  and the  corresponding reflected  process
 above its minimum $H^\theta=(H^\theta(t), t\in \R_+)$:
\[
W_t^\theta=\sqrt{\frac{2}{\beta}} W_t - 2\theta t
\quad\text{and}\quad
H^\theta(t)= W_t^\theta -\inf_{s\in [0,t]} W_s^\theta.
\]
We deduce from equation $(1.7)$ in \cite{dlg:rtlpsbp} that  $H^\theta$ is 
the height process associated to
the branching  mechanism $\psi_\theta$. For a function $H$, we set:
\[
\max(H)=\max(H(t), t\in \R_+).
\]
Let  $\rN^{\psi_\theta} [dH^\theta]$ be
the excursion measure of  $H^\theta$  above  $0$  normalized  such that
$\rN^{\psi_\theta}[\max(H^\theta)\geq r]= c^\theta(r)$. Let
$(L_t^x(H^\theta), t\in \R_+, x\in \R_+)$ be the local time of
$H^\theta$ at time $t$ and level $x$. Let
$\zeta=\inf\{t>0; H^\theta(t)=0\}$ be the duration of the excursion
$H^\theta$ under $\rN^{\psi_\theta} [dH^\theta]$. 
We recall
that $(L_{\zeta}^r(H^\theta), r\in \R_+)$ under $\rN^{\psi_\theta} $
is distributed as $Y^\theta$ under $\N^{\psi_\theta}$. 
From now on we shall identify $Y^\theta$ with $(L_{\zeta}^r(H^\theta),
r\in \R_+)$ and write  $\N^{\psi_\theta}$ for $\rN^{\psi_\theta} $. When
there is no confusion, we shall write $H$ for $H^\theta$ and $Y$ for
$Y^\theta$.  We now recall the construction of the genealogical tree of
the CB process  $Y$ from $H$. 

Let $f$ be a continuous non-negative function defined on $[0,+\infty)$,
such that $f(0)=0$, with compact support. We set $\zeta^f=\sup\{t;
f(t)>0\}$, 
with the convention $\sup\emptyset=0$. Let $d^f$ be the non-negative
function defined by:
\[ 
d^f(s,t) = f(s) + f(t) - 2 \inf_{u\in [s\wedge t , s\vee t]} f(u). 
\]
It can be easily checked that $d^f$ is a semi-metric on $[0,\zeta^f]$. One can
define the equivalence relation associated to $d^f$ by $s\sim t$ if and only if
$d^f(s,t)=0$. Moreover, when we consider the quotient space $
T^f=[0,\zeta^f]/_{ \sim} $
and, noting again $d^f$ the induced metric on $T^f$ and rooting $T^f$ at
$\emptyset^f$, the equivalence class of 0, it can be checked that the
space $(T^f,d^f,\emptyset^f)$ is a compact rooted real tree. 

The so-called  genealogical tree of the CB process $Y$  is the real tree
$\Tau=(T^H, d^H, \emptyset^H)$. In what follows, we shall mainly present
the result using the height process $H$ instead of the genealogical tree.

\subsection{Asymptotics for the number of ancestors}
\label{sec:M}

Let $a>0$ and $(H_k, k\in \ck_a)$ be the excursions of $H$ above level
$a$. It is well known that $\sum_{k\in \ck_a} \delta_{H_k}$ is under
$\N^{\psi_\theta} $ 
conditionally on $(Y_r, r\in [0,a])$ a PPM with intensity
$Y_a \N^{\psi_\theta} [dH]$.  Let $b>a>0$. We define the number $R_{a,
  b}$ of ancestors at time $a$ of the population
living at time $b$ as the number of excursions  above level $a$ which
reach level $b $:
\[
R_{a, b} = \sum_{k \in \ck_a} \ind_{\{\max(H_k) \geq b - a\}}.
\]
To emphasize the dependence of $Y$ and $R$ in $H$, we may write $Y_a(H)$
and $R_{a,b}(H)$.

We give the joint distribution of $(Y_a, Y_b, R_{a,b})$. 

\begin{lem}
   \label{lem:YabRab}
Let $0<a<b$. For $\lambda,\rho,\eta \geq 0$, we have:
\[
\N^{\pthe}\left[1 - \expp{- \rho Y_a-\lam Y_{b} - \eta R_{a,b} }\right]=
u^\theta(\rho + \gamma^\theta_{b-a}(\lambda,\eta), a), 
\]
with
\[
\gamma^\theta_r(\lambda,\eta)
 = (1 - \expp{-\eta}) \cthe r
+ \expp{-\eta} u^{\theta}(\lam,r).
\]
\end{lem}
\begin{proof}
We have:
\begin{align*}
\N^{\pthe}\left[1 - \expp{- \rho Y_a-\lam Y_{b} - \eta R_{a,b}}\right]
&=\N^{\pthe}\left[1 - \expp{ - \rho Y_a - \sum_{k\in \ck_a} (\lambda
    Y_{b-a}(H_k)  
+ \eta\ind_{\{\max(H_k)\geq  b-a\}})} \right]\\
&=\N^{\pthe}\left[1 - \expp{ - Y_a \left(\rho+ \N^{\psi_\theta} \left[1-
      \exp( -\lambda Y_{b-a} - \eta \ind_{\{\zeta\geq  b-a\}})
    \right]\right)}\right]\\
&=\N^{\pthe}\left[1 - \expp{ - Y_a
    \left(\rho+\gamma^\theta_{b-a}(\lambda,\eta)\right)}\right]\\ 
&= u^\theta(\rho + \gamma^\theta_{b-a}(\lambda,\eta), a),
\end{align*}
where we   used the property of the PPM $\sum_{k\in \ck_a} \delta_{H_k}$
for the second equality, and 
\[
1- \expp{-\lambda Y_s - \eta \ind_{\{\zeta\geq s \}}}
= (1-\expp{-\eta} ) \ind_{\{\zeta\geq s \}} + \expp{-\eta} (1-
\expp{-\lambda Y_s} )
\]
as  $Y_s = 0$ on $\{\zeta < s\}$, for the third equality. 
\end{proof}

In order to describe the genealogical structure of $\cz$, following the
beginning of Section \ref{sec:defZ}, we shall consider under $\rP$ the  
PPM:
$$
\sum_{i \in I} \delta_{(t_i, \theta_i, H^i)}(dt, d \theta, dH),
$$
with intensity $2 \beta dt  \mu(d \theta) \N^{\pthe}[dH]$. 

Let  $s > 0$. We define the super-process for the number of ancestors
for the population at time $0$ of
each family:
\[
\cm_s(d\theta) = \sum_{i \in I} \ind_{\{t_i < -s\}} R_{-s - t_i,
  -t_i}(H^{i}) \, \delta_{\theta_i}(d\theta).
\]
Then the number of ancestors at time $-s$ of the population at time 0, 
is:
\[
M_s =\langle \cm_s, 1 \rangle.
\]

We will first consider the joint distribution of $(\cz_0, \cz_{-s}, \cm_s)$.
\begin{prop} 
\label{prop:ZZM-Laplace}
Let $\rho$, $\lambda$ and $\eta$ be non-negative measurable functions
defined on $\R_+$. We have:
\[
\rE\left[\expp{- \langle \cz_{-s}, \rho \rangle - \langle \cz_0, \lam \rangle 
- \langle \cm_s, \eta \rangle }\right]
= \exp\left(-2  \int_0^\infty \mu(d\theta) \,  \left(\log\left(1+
      \lambda(\theta)     \Delta^\theta _s \right)+\log \left(1+ w^\theta(s)
    \right)\right) \right),
\]
with for $r>0$:
\[
w^\theta(s)=\frac{\rho(\theta)+
\gamma^\theta_s(\lambda(\theta),\eta(\theta)) }{2\theta}\cdot
\]
\end{prop}
In particular, we deduce that a.s.:
\begin{equation}
   \label{jointzmz}
\rE\left[\expp{- \langle \cz_0, \lam \rangle 
- \langle \cm_s, \eta \rangle }|\cz_{-s} \right]
= \expp{-2 \int_0^\infty \mu(d\theta)
  \log(1+\lambda(\theta) \Delta^\theta_s)}
  \expp{-\langle \cz_{-s},
 \gamma_s^{\cdot}(\lambda(\cdot), \eta(\cdot)) \rangle }. 
\end{equation}

\begin{proof}
   We have:
\begin{align*}
\rE\left[\expp{- \langle \cz_{-s}, \rho \rangle - \langle \cz_0, \lam \rangle 
- \langle \cm_s, \eta \rangle }\right]   
&=A* \rE\left[\expp{-\sum_{i \in I} \ind_{\{-s < t_i \leq
      0\}}\lambda(\theta_i) Y_{-t_i}^i}\right]\\
&= A \expp {-2 \beta \int_0^s dt \int_0^\infty \mu(d \theta)
  u^{\theta}(\lam(\theta), t)},
\end{align*}
with 
\[
A= \rE\bigg[\exp\Big(-\sum_{i \in I} \ind_{\{t_i \leq -s\}}(\rho(\theta_i) Y_{-s -
  t_i}(H^i) + \lam(\theta_i) Y_{-t_i}(H^i)
+ \eta(\theta_i) R_{-s - t_i, -t_i}(H^i) )\Big)\bigg].
\]
Using Lemma \ref{lem:YabRab}, we get:
\begin{align*}
A
&= \exp \left(-2 \beta \int_0^\infty  dt \int_0^\infty \mu(d\theta) 
 \; \N^{\pthe}[1 - \exp(- \rho(\theta) Y_t-\lam(\theta) Y_{t + s}
- \eta(\theta) R_{t,t + s} )]\right)\\
&=  \exp \left(-2 \beta \int_0^\infty  dt \int_0^\infty \mu(d\theta) 
 \; u^\theta(2\theta w^\theta(s),t)\right).
\end{align*}
Then use \reff{int-u} to end the proof. 
\end{proof}

\begin{rem}
   \label{rem:cvPM}
We get from Proposition \ref{prop:ZZM-Laplace}, 
with  $\rho(\theta)=\rho_0$, $\eta(\theta)=\beta s \eta_0$ and
$\lambda=0$, that:
\[
\rE\left[\expp{- \rho_0 Z_{-s} - \beta s \eta_0 M_s }\right]
= \exp\left(-2  \int_0^\infty \mu(d\theta) \,  \log\left(1+
    \frac{\rho_0+ (1-\expp{-\beta s \eta_0})c^\theta (s)}{2\theta}
  \right)\right). 
\]
We get that:
\[
\lim_{s\downarrow 0} \rE\left[\expp{- \rho_0 Z_{-s} - \beta s \eta_0 M_s
  }\right]
=  \expp  {-2  \int_0^\infty \mu(d\theta) \,  \log\left(1+
    \frac{\rho_0+ \eta_0}{2\theta}
  \right)}= \rE\left[\expp{- (\rho_0+\eta_0) Z_{0}  }\right].
\]
We deduce then from the continuity of the process $Z$ 
that the convergence  $\lim_{s \rightarrow 0+} \beta s M_s = Z_0$ holds
in probability. 
\end{rem}

In fact, we shall see that the convergence in the previous Remark is
an a.s. convergence. 

\begin{prop}
   \label{prop:cv-as-sM}
We have a.s.:
\[
\lim_{s \rightarrow 0+} \beta s M_s = Z_0.
\]
\end{prop}
Notice the order of $M_s$ do not depend on $\mu$: the non-neutral mutations
do not change the asymptotics of the number of ancestors. 
\begin{proof}
Let $\lambda=0$ and $\eta(\theta)=\eta_0$. We deduce from
\reff{jointzmz} that:
\[
\rE[\expp{ -\eta M_s } | \cz_{-s}]
= \expp{-\langle \cz_{-s}, \gamma_s(0,\eta_0) \rangle}
= \expp{-(1 - \expp{-\eta})\sum_{t_i \leq -s} c^{\theta_i}(s) Y^i_{-s -
    t_i}}.
\]
Therefore, conditionally on $\cz_{-s}$, the number of ancestors $M_s$ is
a   Poisson  random   variable   with  mean   $W_s=\sum_{t_i  \leq   -s}
c^{\theta_i}(s) Y^i_{-s - t_i}$. 

We first prove that a.s.:
\begin{equation}
   \label{eq:cvW}
\lim_{s \rightarrow 0+} \beta s
  W_s=Z_0. 
 \end{equation} 

Notice that the function $x\mapsto x/(\expp{x}-1)$ is decreasing on
$(0,+\infty )$ and is less than 1. We deduce that $\beta s
c^\theta(s)\leq 1$ and therefore $\beta s W_s\leq Z_{-s}$. 

Let $q>0$ and set $Z^q_t=\sum_{i\in I} Y^i_{t-t_i} \ind_{\{\theta_i\leq
  q\}}$.   By
construction, we  have a.s. $Z_0^q\leq Z_0$ and
$\lim_{q\rightarrow+\infty } Z^q_0=Z_0$.   Let $\varepsilon\in (0,1)$. There
exists $q$ such that $Z^q_0\geq (1-\varepsilon) Z_0$. 

There exists $s_0>0$ such that for all $\theta\in (0,q]$, $s\in
(0,s_0)$, we have 
$\beta s
c^\theta(s)\geq 1-\varepsilon$. We deduce that, for $s\in (0,s_0)$:
\[
(1-\varepsilon) 
Z^q_{-s}\leq \beta s W_s\leq Z_{-s}.
\]

The process $Z^q=(Z^q_t, t\in \R) $ is distributed as $Z$ with the
mutation rate measure $\mu^q(d\theta)=\mu(d\theta) \ind_{\{\theta\leq
q\}}$ instead of $\mu$.
Since conditions \reff{eq: con-mu} hold for $\mu$, they also hold
for $\mu^q$. In particular the process $Z^q$ is continuous. We deduce
that a.s.:
\[
(1-\varepsilon)^2 Z_0 \leq (1-\varepsilon) Z^q_0 \leq
\liminf_{s\rightarrow 0+} \beta s W_s
\leq  \limsup_{s\rightarrow 0+} \beta s W_s
\leq Z_0.
\]
Since  $\varepsilon\in  (0,1)$  is  arbitrary, this  implies  that  a.s.
\reff{eq:cvW} holds.  

Recall  that $M_s$ is increasing, is conditionally on
$\cz_{-s}$ a Poisson random variable with mean $W_s$. Then use \reff{eq:cvW} and
properties of Poisson distributions to get that a.s.: 
\[
\lim_{s \rightarrow 0+} \beta s M_s/\beta s W_s=1.
\]
This and \reff{eq:cvW} end the proof. 
\end{proof}

We have the following partial result on the fluctuations. 

\begin{theo}
   \label{theo:fluct}
We have the convergence in distribution of $((Z_0,
(Z_0-Z_{-s})/\sqrt{\beta s}), s\geq 0)$ towards $(Z_0, \sqrt{2 Z_0}
\,G)$ with $G$ a standard Gaussian random variable independent of $Z_0$.

Under the conditions:
\begin{equation}
   \label{eq:cond-fluc}
\lim_{A\rightarrow+\infty }
\sqrt{A}\int _A^\infty  \frac{\mu(d\theta)}{\theta}=0
\quad\text{and}\quad
\lim_{A\rightarrow+\infty }
\inv{\sqrt{A}}\int_0^A\mu(d\theta)=0,
\end{equation}
we    have    the     convergence    in    distribution    of    $((Z_0,
(Z_0-Z_{-s})/\sqrt{\beta  s},  (\beta  s M_s  -Z_{-s})/\sqrt{\beta  s}),
s\geq 0)$  towards $(Z_0, \sqrt{  Z_0} \,(G+G'), \sqrt{ Z_0}  \,G)$ with
$G$  and  $G'$  two   independent  standard  Gaussian  random  variables
independent of $Z_0$.
\end{theo}
If conditions
\reff{eq:cond-fluc} do not hold, we may  have a very different
behavior, see the stable case in Section \ref{sec:stable-mut}. 

\begin{proof}
Let $\eta, \lambda\geq 0$, $s\in (0, \min(1/2,1/(4\beta \eta^2)))$ and $\rho>
\max(\lambda,\eta)/2\sqrt{\beta s}$. 
   We deduce from Proposition \ref{prop:ZZM-Laplace} that:
\begin{multline}
\label{eq:exp}
\rE\left[\exp{\left(- \rho Z_{-s} - \frac{\lambda}{\sqrt{\beta
        s}}\left(Z_0-Z_{-s}\right) -  \frac{\eta}{\sqrt{\beta
        s}}\left(\beta s M_s -Z_{-s}\right) \right)}\right]\\
= \exp\left(-2  \int_0^\infty \mu(d\theta) \, \log\left(1+
      \frac{\sigma}{2\theta} + \Lambda_s(\theta)  \right)\right) ,   
\end{multline}
with 
\[
\sigma=\rho - (\lambda^2 + \lambda \eta
        +\frac{\eta^2}{2}), 
\]
\[
\Lambda_s(\theta) = \frac{\lambda(\lambda+\eta)}{2\theta} \left(1-
  \frac{\Delta^\theta_s}{ \beta s} \right) +
\frac{\rho\lambda\sqrt{\beta s} }{2\theta}
\frac{\Delta^\theta_s}{\beta s }
+ \frac{\eta}{2\theta\sqrt{\beta s}} A(\eta \sqrt{\beta s }, 2\beta
\theta s) ,
\]
and
\[
A(x,y)= \frac{1-\expp{-x}}{x}\frac{y}{\expp{y} -1} - 1 + \frac{x}{2}\cdot
\]
We have for $y>0$ and $x\in (0, 1/2)$:
\[
- \min(1,y)\leq \left(1-\frac{x}{2} \right)\left(\frac{y}{\expp{y} -1}-1 \right)\leq
A(x,y) \leq  \frac{x^2}{6}\cdot
\]
Recall that for $x\geq 0$, we have $x -x^2/2\leq 1-\expp{-x}\leq x$. 
 We deduce that:
\begin{equation}
   \label{eq:majo}
\Lambda_s(\theta) \leq  \frac{\lambda(\lambda+\eta)}{2\theta} \min(1,
\beta \theta s)+
\frac{\rho\lambda\sqrt{\beta s} }{2\theta} + \frac{\eta^3 \sqrt{\beta s}
}{12\theta}
\leq  \frac{\lambda(\lambda+\eta)+(\rho \lambda +\eta^3) \sqrt{\beta} }{2\theta}\cdot
\end{equation}
We also have:
\begin{equation}
   \label{eq:mino}
\Lambda_s(\theta) \geq  - 
\frac{\eta}{2\theta\sqrt{\beta s}}  \min(1, 2\beta\theta    s)
\geq  - 
\eta \min\left(\inv{2\theta\sqrt{\beta s}} , \sqrt{\beta s}
\right). 
\end{equation}
In particular, we have:
\begin{equation}
   \label{eq:limA}
\lim_{s\rightarrow 0} \Lambda_s(\theta)=0. 
\end{equation}

Let $M_0>0$ large, $\varepsilon_0>0$  small, and $s_0>0$ small enough such
that  $\varepsilon_0   -  M_0^2  \sqrt{\beta   s_0}>0$  and  $M_0\sqrt{\beta
  s_0}<1/2$.  Set $I_0=[0,M_0]^3\bigcap  \{\rho- (\lambda^2+\lambda  \eta +
(\eta^2/2))>\varepsilon_0\}$. Notice that $\Lambda_s(\theta) $ is analytic in
$(\rho,\lambda,\eta)$. 
We deduce that the integral
\[
F_s(\rho,\lambda,\eta)=\int_0^\infty \mu(d\theta) \,  \log\left(1+
      \frac{\sigma}{2\theta} + \Lambda_s(\theta)  \right)
\]
is   well   defined   in   $(\rho,\lambda,\eta)\in I_0$
for all $s\in (0,s_0]$.  For fixed $\rho>\varepsilon_0>0$, it is not
difficult to check that there exists $s_1>0$ smaller than $s_0$ such
that for $s\in (0, s_1]$, $F_s(\rho,\lambda,\eta)$ is also analytic in
$(\lambda,\eta)$ such that
$\lambda^2+\lambda\eta+\eta^2/2<\rho-\varepsilon_0$.  We
deduce that \reff{eq:exp} is valid for $\rho> \lambda^2+\lambda \eta +
(\eta^2/2)$ and $s$ small (and not only for 
$\rho>\max(\lambda,\eta)/2\sqrt{\beta s}$).\\

If $\eta=0$, we have $0\leq \Lambda_s(\theta) \leq  \lambda(\lambda+
\rho \sqrt{\beta})/(2\theta)$. By dominated convergence, we get, using
\reff{eq:limA}, that for $\rho>\lambda^2$:
\[
\lim_{s\rightarrow 0} \int_0^\infty \mu(d\theta) \,  \log\left(1+
      \frac{\rho - \lambda^2 }{2\theta} + \Lambda_s(\theta)  \right)
= \int_0^\infty \mu(d\theta) \,  \log\left(1+
      \frac{\rho - \lambda^2}{2\theta} \right) .
\]
Thanks to \reff{eq:exp}, we obtain that for $\rho>\lambda^2$:
\[
\lim_{s\rightarrow 0} 
\rE\left[\expp{- \rho Z_{-s} - \frac{\lambda}{\sqrt{\beta
        s}}\left(Z_0-Z_{-s}\right) }\right]
= \expp{-2 \int_0^\infty \mu(d\theta) \,  \log\left(1+
      \frac{\rho - \lambda^2}{2\theta} \right)}
= \rE[\expp{- \rho Z_0 -\lambda \sqrt{2 Z_0}\, G}],  .
\]
with $G$ a standard Gaussian random variable independent of $Z_0$. 
This implies the convergence in distribution of the sequence $((Z_{-s},
(Z_0-Z_{-s})/\sqrt{\beta s}), s\geq 0)$ towards $(Z_0, \sqrt{2 Z_0}
G)$. Then use that $Z$ is continuous to get the first part of the
Theorem. \\

We assume $\eta>0$.  Let $(\rho,\lambda,\eta) \in 
I_0$. Set 
\[
F_0(\rho,\lambda,\eta)=\int_0^\infty \mu(d\theta) \,  \log\left(1+
      \frac{\sigma}{2\theta}  \right).
\]
We shall prove that under \reff{eq:cond-fluc}, we have $
\lim_{s\rightarrow 0} F_s(\rho,\lambda,\eta)=F_0(\rho,\lambda,\eta)$. 
Notice that $\sigma>\varepsilon_0>0$. We have:
\[
H_s=\val{F_s(\rho,\lambda,\eta) -F_0(\rho,\lambda,\eta)}
\leq \int_0^\infty \mu(d\theta) \,
\val{\log\left(1+\frac{2\theta\Lambda_s(\theta)}{\sigma+ 2\theta}  \right)}.
\]
We shall denote by $C_k$ for $k\in \N$ some finite positive constants which
depend only on $M_0, \varepsilon_0$ and $s_0$. 
We deduce from \reff{eq:majo} and  \reff{eq:mino} that for all
$(\rho,\lambda,\eta)\in I_0$ and $s\in (0,s_0]$: 
\[
\val{\Lambda_s(\theta)}\leq C_1\left(\frac{\sqrt{s}}{\theta}+
 \min\left(\inv{\theta\sqrt{s}}, \sqrt{s}\right)\right).
\]
Thus, there exists $s_1>0$ small enough and less than $s_0$ such that for all
$(\rho,\lambda,\eta)\in I_0$ and $s\in (0,s_1]$: 
\[
\val{\log\left(1+\frac{2\theta\Lambda_s(\theta)}{\sigma+ 2\theta}  \right)}
\leq  C_2 \sqrt{s}\ind_{[0,1/s]}(\theta)+
 C_3 \inv{\theta \sqrt{s}}\ind_{[1/s,+\infty )}(\theta).
\]
Then, we deduce from \reff{eq:cond-fluc} that $\lim_{s\rightarrow 0}
H_s=0$. This implies that:
\begin{align*}
\lim_{s\rightarrow 0}
\rE\left[\expp{- \rho Z_{-s} - \frac{\lambda}{\sqrt{\beta
        s}}\left(Z_0-Z_{-s}\right) -  \frac{\eta}{\sqrt{\beta
        s}}\left(\beta s M_s -Z_{-s}\right) }\right]
&= \expp{-2 \int_0^\infty \mu(d\theta) \,  \log\left(1+
      \frac{\sigma }{2\theta} \right)}\\
&= \rE\left[\expp{- \sigma Z_0}\right]\\
&= \rE\left[\expp{- 
(\rho - (\lambda^2 + \lambda \eta
        +\frac{\eta^2}{2})) Z_0}\right]\\
&=\rE\left[\expp{- \rho Z_0
    -\lambda \sqrt{Z_0}\, G - (\lambda+\eta)\sqrt{Z_0}\, G'}\right],
\end{align*}
with $G'$ distributed as $G$ and independent of $Z_0$ and $G$. 
This gives the second part of the Theorem. 
\end{proof}

We have the following representation of the limit in Theorem
\ref{theo:fluct}.

\begin{lem}
   \label{lem:Z-Z} Let $G$ be a standard Gaussian random variable
   independent of $Z_0$. 
We have that $\sqrt{Z_0}\, G$ is distributed as $Z'_0-Z''_0$, with $Z'_0$
and $Z''_0$ independent and distributed as $Z_0$ with mutation rate
measure $\mu'$ defined by $\langle \mu', \varphi \rangle=\int_0^\infty
\mu(d\theta) \varphi(\theta^2)$. 
\end{lem}
\begin{proof}
   Assume first there exists $\theta_0>0$ such that
   $\mu((0,\theta_0])=0$. Then, we deduce from \reff{eq:lap-Z} that
   $Z_0$ has positive exponential moments, that is \reff{eq:lap-Z}
   holds for $\lambda\geq -\theta_0$. We obtain for $\val{\lambda}\leq
   \theta_0$: 
\[
\rE[\expp{-\lambda (Z'_0- Z''_0)}]
= \exp\left(-2    \int_0^{\infty}
  \mu'(d\theta)\, \log\left(1-\frac{\lambda^2}{4\theta^2}\right)\right)
= \rE[\expp{-\lambda \sqrt{ Z_0}\, G}].
\]

To conclude, use that $Z_0$ is the limit in distribution of $Z_0$ 
with mutation rate
measure $\ind_{\{\theta\geq \theta_0\}}\mu(d\theta)$ as $\theta_0$ goes
down to $0$. 
\end{proof}

\section{Stable mutation rate measure}
\label{sec:stable-mut}

Let $c>0$ and $\alpha\in (0,1)$. 
In this Section, we will consider the stable mutation rate: 
\[
\mu(d\theta) = c  \theta^{\alpha - 1} \, \ind_{\{\theta > 0\}}\, d\theta.
\]
Notice that $\mu$ satisfies conditions \reff{eq: con-mu}. In addition, notice that
$\rE[Z_0]=+\infty $ and that
$\langle \mu,1 \rangle=+\infty $ which in turn implies that $\{t;  Z_t  =  0\}  =
  \emptyset$ a.s. thanks to Proposition \ref{prop:Zt=0}.

\subsection{Bottleneck effect}

We present a drastic bottleneck effect which was not observed
in~\cite{cd:spstsbp}.  
\begin{lem}
We have $\rE[Z^A] = +\infty$ if  $\alpha\in (0, 1/2]$ and 
$\rE[Z^A] < +\infty$ if  $\alpha\in (1/2, 1)$.
\end{lem}
In the case  $\alpha\in (1/2, 1)$,  we observe a  drastic
bottleneck effect as  $\rE[Z^A]/\rE[Z_0] = 0$.  

\begin{proof}
By \eqref{expeza}, we have:
\begin{equation*} 
\rE[Z^A]=2 \beta c^2 \int_0^\infty dt \int_0^\infty \frac{2\theta^\alpha}{\ebt - 1}d\theta 
\int_0^\infty \frac{1 - \expp{- 2 \beta \theta' t}}{\theta'^{2 - \alpha}} d\theta'
\expp{2 c \int_0^\infty \log(1 - \expp{-2 \beta \theta'' t})\theta''^{\alpha - 1} d\theta''}.
\end{equation*}
We get:
\[
2 \beta c^2  \int_0^\infty \frac{2\theta^\alpha}{\ebt -
  1}d\theta = C_1 t^{-1-\alpha}
\quad\text{with}\quad
C_1 =4\beta c^2 \int_0^\infty
\frac{x^\alpha}{\expp{2\beta x}  -
  1}dx, 
\]
\[
2\beta c^2 \int_0^\infty  \frac{1 - \expp{- 2 \beta \theta' t}}{\theta'^{2 - \alpha}}
d\theta' = C_2 t^{1 - \alpha}
\quad\text{with}\quad
C_2 = 2\beta c^2 \int_0^\infty\frac{1 - \expp{-2 \beta x}}{x^{2 - \alpha}} dx,
\]
and
\begin{equation}
   \label{eq:C3}
2 c \int_0^\infty \log(1 - \expp{-2 \beta \theta'' t})\theta''^{\alpha - 1} 
d\theta'' = -C_3 t^{-\alpha} 
\quad\text{with}\quad
C_3 = -2 c \int_0^\infty \log(1 - \expp{-2 \beta x})x^{\alpha - 1}\,dx.
\end{equation}
Notice that $C_1$, $C_2$  and $C_2$ are positive finite constants. 
We deduce that:
\[
\rE[Z^A] 
= C_1 C_2 \int_0^\infty dt\,\,  t^{ - 2\alpha} \expp{-C_3 t^{-\alpha}}
= \frac{C_1C_2}{\alpha}\int_0^\infty dr\,\,  r^{ 1- \inv{\alpha}}
\expp{-C_3 r}. 
\]
This implies that  $\rE[Z^A] = +\infty$ if  $\alpha\in (0, 1/2]$ and
that 
$\rE[Z^A] < +\infty$ if  $\alpha\in (1/2, 1)$.
\end{proof}

\subsection{Type of the MRCA and type of a random individual}

Recall  from Section  \ref{sec:prelim}  that a  CB process  $Y^\theta$ has  type
$\theta$  and that  $Y^\theta$ is  stochastically larger  than  $Y^q$ if
$\theta\leq q$.  We shall  say that the  type (or mutation)  $\theta$ is
more advantageous than the type $q$. The following proposition states that
the type of  the MRCA is (stochastically) more  advantageous than the type
of an individual taken at random at the current time.

Recall from  Section \ref{sec:defZ}  that the type  of an  individual in
family  $Y^i$  is  $\theta_i$.  We  define the  type  $\Theta_*$  of  an
individual taken  at random at time  $0$ as follows: conditionally on
$\cz$,   $\Theta_*$   is    equal   to   $\theta_i$   with   probability
$Y^i_{-t_i}/Z_0$.

\begin{prop}
   \label{prop:fitness}
We have that $\Theta$ is stochastically smaller than $\Theta_*$: for all
$q\geq 0$, 
\[
\rP(\Theta_* \leq q)\leq  \rP(\Theta\leq  q).
\]
\end{prop}

\begin{proof}
Firstly, we  give the distribution of $\Theta$.  We have for $\theta>0$:
\begin{align*}
   \frac{\rP(\Theta\in d\theta)}{d\theta}
&= c\theta^{\alpha-1} \int_0^\infty dt \; \frac{4\beta  \theta}{\expp{2
    \beta \theta t} - 1} 
\exp\left(2c \int_0^\infty \log(1 - \expp{-2 \beta q
    t}) \, q^{\alpha-1} dq \right)\\
&= 2c\theta^{\alpha-1} \int_0^\infty  \frac{ds}{\expp{s} -1} \expp{-2c
  \, a_1 \theta^\alpha s^{-\alpha}},
\end{align*}
where we used   \reff{eq:joint-athe} for the first equality, the change
of variables $r=2\beta q t$ ($t$ fixed) and $s=2\beta \theta t$ as well
as 
\[
a_1= - \int_0^\infty  \log\left(1-\expp{-r}\right)\, r^{\alpha-1} \, dr
\]
for the second equality. Set $Q=2c \, \Theta^\alpha$ so that for $q>0$:
\[
  \frac{\rP(Q\in dq)}{dq}
=
\inv{\alpha a_1} \int_0^\infty  \frac{s^\alpha ds}{\expp{s} -1}\;  a_1
s^{-\alpha} \expp{-a_1 
  s^{-\alpha} q }.
\]
Then we deduce that $Q$ is distributed as $E S^\alpha/a_1$, where $E$ is an
exponential random variable with mean $1$ independent of the random
variable $S$ whose density is:
\[
f(s)=\inv{\alpha a_1} \frac{s^\alpha }{\expp{s} -1}\;\ind_{\{s>0\}}.
\]

Secondly, we give the distribution of $\Theta_*$. Let $F$ be a non-negative
measurable function, we have:
\begin{align*}
   \rE\left[F(\Theta_*)\right]
&= \rE\left[\sum_{i\in I} \frac{Y^i_{-t_i}}{Z_0} F(\theta_i)\right]\\
&=\int_0^\infty  d\lambda \, \rE\left[\sum_{i\in I}F(\theta_i)
  Y^i_{-t_i}\expp{-\lambda Y^i_{-t_i}} \expp{-\lambda \sum_{j\in I\backslash
      \{i\}} Y^j_{-t_j}}\right]\\
&=2c\beta \int_0^\infty  d\lambda \, \int _0^\infty  dt \int_0^\infty 
\theta^{\alpha-1} d\theta \; F(\theta) \N^{\psi_\theta}\left[Y_t\expp{-\lambda
    Y_t}\right] \rE\left[\expp{-\lambda Z_0}\right]\\
&=2c\beta \int_0^\infty  d\lambda \, \int _0^\infty  dt \int_0^\infty 
\theta^{\alpha-1} d\theta \; F(\theta) \frac{\expp{-2\beta \theta
    t}}{\left(1+\lambda \Delta^\theta_t\right)^2} \expp{-2 c   \int_0^{\infty}
  \log\left(1+\frac{\lambda}{2q}\right) \, q^{\alpha-1} dq},
\end{align*}
where we used the definition of $\Theta_*$ for the first equality, the
PPM properties for the third equality, \reff{eq:mean} and
\reff{eq:lap-Z} for the fourth equality. Set 
\[
a_2= \int_0^\infty  \log\left(1+\frac{1}{2q}\right) \, q^{\alpha-1} dq
\]
and use the change of variables $q =\lambda r$ ($\lambda$ fixed), 
$u=1- \expp{-2\beta \theta t}$ ($\theta$ fixed) and
$s=a \theta/\lambda$ ($\theta$ fixed) with $a^\alpha= a_1/a_2$
and $b=2/a$ to get:
\[
  \rE\left[F(\Theta_*)\right]
= 2c \int_0^\infty 
\theta^{\alpha-1} d\theta \; F(\theta) 
\int_0^\infty  \frac{ds}{s(1+bs)} \expp{- 2c\, a_1 s^{-\alpha}
  \theta^\alpha}.
\]
Set $Q_*=2c\, \Theta_*^\alpha$ and we deduce that:
\[
  \frac{\rP(Q_*\in dq)}{dq}
=
\inv{\alpha a_1} \int_0^\infty  \frac{s^\alpha ds}{s(1+bs)}\;  a_1
s^{-\alpha} \expp{-a_1 
  s^{-\alpha} q }.
\]
Then similarly $Q_*$ is distributed as $E S_*^\alpha/a_1$ with $S_*$
a  random
variable independent of $E$ and  with density:
\[
f_*(s)=\inv{\alpha a_1} \frac{s^\alpha }{s(1+bs)}\;\ind_{\{s>0\}}.
\]

Thirdly,        define        
\[
h(s)=\frac{f_*(s)}{f(s)}=\frac{\expp{s} -1}{s(1+bs)}\;\ind_{\{s>0\}},
\]
so that $\rE[H(S_*)]=\rE[H(S)h(S)]$. A  study of the continuous function
$h$  and   the  condition   that  $\rE[h(S)]=1$  yield   that  $h(0)=1$,
$\lim_{s\rightarrow+\infty } h(s)=+\infty $  and there exists $s_0$ such
that  $h\leq 1$  on $[0,s_0]$  and $h\geq  1$ on  $[s_0, +\infty  )$. We
deduce that $\rP(S_*\leq s)\leq \rP(S\leq s)$ for all $s\geq 0$, that is
$S_*$  is  stochastically  larger  than  $S$. This  implies  that  $Q_*$
(resp. $\Theta_*$) is stochastically larger than $Q$ (resp. $\Theta$).
\end{proof}

\subsection{Number of families}
\label{sec:Nstable}
We  compare the  number  of  families  with the  neutral  case
(quadratic and stable  branching mechanism). Recall $\Lambda(s)$ defined
in  \reff{eq:def-L}. We  deduce from  \reff{eq:C3}  that $\Lambda(s)=C_3
s^{-\alpha}$.  Lemma \ref{lem:cvN/L} implies that a.s.:
\begin{equation}
   \label{eq:saN}
\lim_{s \rightarrow 0+} s^\alpha \, N_s = C_3.
\end{equation}

We  can  compare \reff{eq:saN}  with  the  stationary  stable case  with
immigration,  see \cite{cd:spstsbp},  that  is $\psi(\lam)  = \lam^a+  b
\lam$ with  $1 <  a \leq 2$  and $ b  > 0$.   According to Section  6 in
\cite{cd:spstsbp}, we have that the  number $N^*_s$ of families alive at
time $-s$ which are still alive  at time $0$ is a.s.  equivalent, as $s$
goes down  to 0,  to $\Lambda^*$ defined  by (31)  in \cite{cd:spstsbp}.
Using that:
\[
c(r) = \frac{\expp{-b r}}{[b^{-1}(1 - \expp{-(a - 1) b r})]^{\frac{1}{a
      - 1}}},
\]
see Example 3.1 in \textsc{Li} \cite{l:mvbmp}, it is easy to get that:
\[
\Lambda^* (s) = -\frac{a}{a - 1}\log(1 - \expp{-(a - 1)b s}).
\]
Notice that  $\Lambda^*(s)$ is equivalent to 
 $(a/a-1)|\log(s)|$ as $s$ goes down to 0. This implies that a.s. 
\[
\lim_{s \rightarrow 0+} |\log(s)|^{-1} \, N_s^* = \frac{a-1}{a}\cdot
\] 
Therefore the number  of families for stable case  with neutral mutation
is much smaller than that of CB process with non-neutral mutations (that
is with stable rate of mutation).

\subsection{Fluctuations of the number of ancestors}

We consider the fluctuations of the number of ancestors. Recall
notations from Theorem \ref{theo:fluct}. If $\alpha<1/2$, then
conditions \reff{eq:cond-fluc} hold and the fluctuations of $M_s$  are
given in Theorem \ref{theo:fluct}. For $\alpha\geq 1/2$, conditions
\reff{eq:cond-fluc}  do not hold and the fluctuations
of $M_s$ are given in the next Proposition. We define for $\alpha\geq 1/2$:
\[
h(\alpha)= c2^{1-\alpha} \int_0^\infty  \frac{dq}{q^{2-\alpha} }\, \frac{\expp{q}- 1
    -q}{\expp{q }-1} \cdot
\]

\begin{prop}
   \label{prop:fluct-stab}
Let $\alpha\in [1/2,1)$. We have the following convergence in distribution of $((Z_0,
(\beta s)^{\alpha-1}  (\beta s M_s -Z_{-s}),
s\geq 0)$ towards $(Z_0, \sqrt{ Z_0} \,G -h(1/2))$ if $\alpha=1/2$ and
towards  $(Z_0, -h(\alpha))$ if $\alpha\in (1/2,1)$,  with
$G$ a  standard Gaussian random variable independent of $Z_0$.
\end{prop}
We explain the contribution  of $-h(\alpha)$ as a law of large number
effect produced by the large number  of (small) populations with large
parameter $\theta$. This effect is negligible if conditions
\reff{eq:cond-fluc} hold that is $\alpha\in (0,1/2)$ but is significant
for $\alpha\in [1/2,1)$. 

\begin{proof}
Mimicking the first part of the proof of Theorem \ref{theo:fluct}, we get
that for $\eta \geq 0$, $\rho> \eta^2/2 $ (or $\rho>0$
if $\alpha>1/2)$ and $s>0$ small:
\[
\rE\left[\exp{\left(- \rho Z_{-s} -\eta(\beta
        s)^{\alpha-1}\left(\beta s M_s -Z_{-s}\right) \right)}\right]
= \expp{-2^{1-\alpha} c B(\beta s)},   
\]
with $B$ defined for $t$ small by:
\[
B(t)=2^\alpha \int_0^\infty  \frac{d\theta}{\theta^{1-\alpha}}
\,\log \left( 1+ \frac{\rho}{2\theta} + \frac{\eta t^\alpha}{2\theta t
  } \left(\frac{1-\expp{-\eta t^\alpha}}{\eta t^\alpha}\, \frac{2\theta
      t}{\expp{2 \theta t } -1} -1\right)
 \right).
\]
Use $q=2\theta t$ to get:
\[
B(t)=\int_0^\infty  \frac{dq}{q^{1-\alpha}} D(q,t)
\quad\text{with}\quad 
D(q,t)= t^{-\alpha} \log \left( 1+ \frac{\rho t}{q} + \frac{\eta t^\alpha}{q
  } \left(\frac{1-\expp{-\eta t^\alpha}}{\eta t^\alpha}\,
    \frac{q}{\expp{q} -1} -1\right) \right) .
\]
Notice that $B(t)$ is well defined for $\eta>0$, $\rho>\eta^2
\ind_{\{\alpha=1/2\}}/2$ and $t>0$ small (depending on $\rho,  \eta
$). 

Let $\varepsilon>0$ be small and $a\in (0,\varepsilon)$ such that for all $q\in (0,a]$,
we have:
\begin{equation}
   \label{eq:majo-a}
\val{\inv{q}\left(\frac{q}{\expp{q}-1}-1\right) +\inv{2}}<\frac{\varepsilon}{4}\cdot
\end{equation}

We first consider $q\geq a$. For $t$ small enough (depending on $\rho,
\eta, \varepsilon, a$), we
have for all $q\geq a$: 
\[
\val{D(q,t)}<  \frac{4(\eta+1)}{q}.
\]
Since $\lim_{t\rightarrow 0} D(q,t)=-  \eta (\expp{q}- 1
    -q)/ q(\expp{q} -1)$, we deduce by dominated convergence that:
\begin{equation}
   \label{eq:cvD>a}
\lim_{t\rightarrow 0+} \int_a^\infty  \frac{dq}{q^{1-\alpha}} D(q,t)= 
- \eta\int_a^\infty    \frac{dq}{q^{2-\alpha} }\, \frac{\expp{q}- 1
    -q}{\expp{q }-1} \cdot
\end{equation}

Secondly, we consider $q\in (0,a)$. Notice that:
\[
D(q,t)=t^{-\alpha} \log\left( 1 + E(q,t) + \frac{F(q,t)}{q} \right),
\]
with
\[
E(q,t)=\eta t^\alpha \, \frac{1-\expp{-\eta
      t^\alpha}}{\eta t^\alpha} \, \inv{q} \left(
    \frac{q}{\expp{q} -1} -1\right) 
\quad\text{and}\quad
F(q,t) = \rho t+ \eta t^\alpha \left(\frac{1-\expp{-\eta t^\alpha}}{\eta
    t^\alpha} -1\right) . 
\]
We get that for $t$ small enough (depending on $\rho, \eta,
\varepsilon$) and $q\in (0,a)$:
\begin{equation}
   \label{eq:Del-a>1/2}
\rho t  - \frac{\eta^2}{2} t^{2\alpha} (1+\varepsilon)
\leq  F(q,t) \leq 
\rho t  - \frac{\eta^2}{2} t^{2\alpha} (1-\varepsilon)
\end{equation}
as well as, using \reff{eq:majo-a},
\[
- \frac{\eta t^\alpha}{2} (1+\varepsilon) 
\leq  E(q,t) \leq - \frac{\eta t^\alpha}{2} (1-\varepsilon) .
\]
This implies that:
\begin{equation}
   \label{eq:D+-}
D_{+\varepsilon}(q,t) \leq  D(q,t) \leq  D_{-\varepsilon}(q,t),
\end{equation}
with for given $z$ and $t$ small (depending on $\rho, \eta, z$):
\[
D_z(q,t)=t^{-\alpha} \log\left( 1 -\frac{\eta t^\alpha}{2} (1+z)  +
  \frac{\rho t  - \frac{\eta^2}{2} t^{2\alpha} (1+z) }{q} \right).
\]
Since 
\[
 \int_0^a \frac{dq}{q^{1-\alpha}} \log \left( 1- 
  a_1+  \frac{a_2}{q}
\right)
=\log \left( 1- a_1\right) \int_0^a
\frac{dq}{q^{1-\alpha}}
+ \frac{a_2^\alpha }{(1-a_1)^\alpha} \int_0^{\frac{(1-a_1)}{a_2}a}
\frac{dq}{q^{1-\alpha}} \log \left( 1+ \inv{q} 
\right),
\]
we deduce that for $z\in \{+\varepsilon, -\varepsilon\}$:
\begin{multline*}
\lim_{t\rightarrow 0}    \int_0^a  \frac{dq}{q^{1-\alpha}} \, D_z(q,t)\\
\begin{aligned}
&= - \frac{\eta}{2}(1+z) \int_0^a
\frac{dq}{q^{1-\alpha}} + \left(\rho -
  \frac{\eta^2}{2}(1+z)\ind_{\{\alpha=1/2\}}  \right)^\alpha
\int_0^\infty  \frac{dq}{q^{1-\alpha}} \log \left( 1+ \inv{q} 
\right) \\
&= - \frac{\eta}{2}(1+z) \int_0^a
\frac{dq}{q^{1-\alpha}} +
\int_0^\infty  \frac{dq}{q^{1-\alpha}} \log \left( 1+ \frac{\rho -
  \frac{\eta^2}{2}(1+z)\ind_{\{\alpha=1/2\}} }{q} 
\right). 
\end{aligned}
\end{multline*}
Since $\varepsilon$ can be arbitrarily small and that $a<\varepsilon$,
we deduce from \reff{eq:cvD>a}, \reff{eq:D+-} and the previous
convergence that:
\[
\lim_{t\rightarrow 0}    B(t)= \lim_{t\rightarrow 0}    \int_0^\infty
\frac{dq}{q^{1-\alpha}} \, D(q,t) 
= - \frac{\eta h(\alpha)}{c2^{1-\alpha} } + \int_0^\infty
\frac{dq}{q^{1-\alpha}} \log \left( 1+ \frac{\rho - 
  \frac{\eta^2}{2}\ind_{\{\alpha=1/2\}} }{q} 
\right). 
\]
This implies that:
\begin{multline*}
\lim_{s\rightarrow 0}  
\rE\left[\exp{\left(- \rho Z_{-s} -\eta(\beta
        s)^{\alpha-1}\left(\beta s M_s -Z_{-s}\right) \right)}\right]\\
= \rE\left[\exp{\left(- \left(\rho -
  \frac{\eta^2}{2}\ind_{\{\alpha=1/2\}} \right) Z_0  +
 h(\alpha)\eta \right)}\right]. 
\end{multline*}
This and the continuity of $Z$ give the result. 
\end{proof}

\noindent
\textbf{Acknowledgment}: H. Bi would like to express his gratitude to J.-F. 
Delmas for his help during the stay at CERMICS.

\end{document}